\documentclass[11pt]{amsart}
\usepackage[margin=3cm]{geometry}

\usepackage[utf8]{inputenc}
\usepackage{mathtools,amsfonts,amsthm,mathrsfs,amssymb}
\usepackage{textcomp}
\mathtoolsset{showonlyrefs}
\usepackage{bookmark}
\usepackage{microtype}
\usepackage{todonotes}

\newtheorem{thm}{Theorem}
\newtheorem{lemma}{Lemma}

\newtheorem{conj}[thm]{Conjecture}
\theoremstyle{definition}
\newtheorem*{definition*}{Definition}

\newtheorem*{acknowledgement}{Acknowledgments}

\newcommand{\cT}{\mathcal{T}}
\newcommand{\cH}{\mathcal{H}}
\newcommand{\bH}{\mathbf{H}}

\newcommand{\esp}[1]{\mathbb{E}\left(#1\right)}
\newcommand{\pth}[1]{\left(#1\right)}
\newcommand{\sset}[2]{\left\{#1~\middle|~#2\right\}}
\newcommand{\ee}{\mathrm{e}}

\DeclareMathOperator{\et}{and}

\renewcommand{\le}{\leqslant}
\renewcommand{\leq}{\leqslant}
\renewcommand{\ge}{\geqslant}
\renewcommand{\geq}{\geqslant}

\title{Packing and covering balls in graphs excluding a minor}

\author[N. Bousquet]{Nicolas Bousquet}
\address{Laboratoire G-SCOP, CNRS, Univ. Grenoble Alpes, Grenoble, France.}
\email{nicolas.bousquet@grenoble-inp.fr}

\author[W. CvB]{Wouter Cames van Batenburg}
\address{D\'epartement d'Informatique, Universit\'e Libre de Bruxelles, Brussels, Belgium.}
\email{wcamesva@ulb.ac.be}

\author[L. Esperet]{Louis Esperet}
\address{Laboratoire G-SCOP, CNRS, Univ. Grenoble Alpes, Grenoble, France.}
\email{louis.esperet@grenoble-inp.fr}

\author[G. Joret]{Gwena\"el Joret}
\address{D\'epartement d'Informatique, Universit\'e Libre de Bruxelles, Brussels, Belgium.}
\email{gjoret@ulb.ac.be}

\author[W. Lochet]{William Lochet}
\address{Algorithms Research Group, University of Bergen, Bergen, Norway.}
\email{william.lochet@gmail.com}

\author[C. Muller]{Carole Muller}
\address{D\'epartement de Math\'ematique, Universit\'e Libre de Bruxelles, Brussels, Belgium.}
\email{camuller@ulb.ac.be}

\author[F.\ Pirot]{Fran\c{c}ois Pirot}
\address{D\'epartement d'Informatique, Universit\'e Libre de Bruxelles, Brussels, Belgium, and Laboratoire G-SCOP, CNRS, Univ. Grenoble Alpes, Grenoble, France.}
\email{francois.pirot@grenoble-inp.fr}

\thanks{N.\ Bousquet is supported by ANR Project DISTANCIA (\textsc{ANR-17-CE40-0015}). 
W.\ Cames van Batenburg, G.\ Joret, and F.\ Pirot are supported by an ARC grant from the Wallonia-Brussels Federation of Belgium. 
L.\ Esperet is supported by ANR Projects GATO
(\textsc{ANR-16-CE40-0009-01}) and GrR (\textsc{ANR-18-CE40-0032}). 
C.\ Muller is supported by the Luxembourg National Research Fund (FNR) (Grant Agreement Nr 11628910).} 

\begin{document}
\begin{abstract}
We prove that for every integer $t\ge 1$ there exists a constant $c_t$ such that for every $K_t$-minor-free graph $G$, and every set $S$ of balls in $G$, the minimum size of a set of vertices of $G$ intersecting all the balls of $S$ is at most $c_t$ times the maximum number of vertex-disjoint balls in $S$. This was conjectured by Chepoi, Estellon, and Vax{{\`e}}s in 2007 in the special case of planar graphs and of balls having the same radius.
\end{abstract}
\maketitle

\section{Introduction}

A hypergraph $\cH$ is a pair $(V,\mathcal{E})$ where $V$ is the vertex set and $\mathcal{E}\subseteq 2^V$ is the edge set of $\cH$.
A \emph{matching} in  a hypergraph $\cH$ is a set of pairwise vertex-disjoint edges, and a \emph{transversal} is a set of vertices that intersects every edge.
The \emph{matching number} of a hypergraph $\cH$, denoted by $\nu(\cH)$, is the maximum number of edges in a matching. The \emph{transversal number} of $\cH$, denoted by $\tau(\cH)$, is the minimum size of a transversal of $\cH$. We can also consider the linear relaxation of these two parameters: we define the \emph{fractional matching number} $\nu^*(\cH)$ and the \emph{fractional transversal number} $\tau^*(\cH)$ as follows.
\begin{align*}
\nu^*(\cH) &= \max \sum_{e\in \mathcal{E}(\cH)} w_e \\
\mbox{given that } & \begin{cases}
\sum_{e \ni v}\limits w_e \le 1 &\quad \mbox{ for every vertex } v \mbox{ of }\cH  \\
w_e \ge 0 &\quad \mbox{ for every edge } e \mbox{ of }\cH,
\end{cases}
\end{align*}
and the dual of this linear program is
\begin{align*}
\tau^*(\cH) &= \min \sum_{v\in V(\cH)} w_v \\
\mbox{given that } & \begin{cases}
\sum_{v \in e}\limits w_v \ge 1 &\quad \mbox{ for every edge } e \mbox{ of } \cH \\
w_v \ge 0 &\quad \mbox{ for every vertex } v \mbox{ of } \cH.
\end{cases}
\end{align*}

By the strong duality theorem, $\nu(\cH)\le \nu^*(\cH)=\tau^*(\cH)\le \tau(\cH)$ for every hypergraph $\cH$. Given a class $\mathcal{C}$ of hypergraphs, a classical problem in combinatorial optimization is to decide whether there exists a function $f$ such that $\tau(\cH)\le f(\nu(\cH))$ for every $\cH\in \mathcal{C}$. If this is the case the class $\mathcal{C}$ is sometimes said to have the \emph{Erd\H os-P\'osa property}. Classical examples include the family of all cycles of a graph~\cite{ErPo65} (i.e.\ given a graph $G=(V,E)$ we consider the hypergraph with vertex set $V$ whose edges are all the cycles of $G$), and the family of all directed cycles of a directed graph~\cite{RRST96}. A desirable property is the existence of a constant $c$ such that $\tau(\cH)\le c\cdot \tau^*(\cH)$ or $\nu^*(\cH)\le c\cdot \nu(\cH)$, or (even better)
$\tau(\cH)\le c\cdot \nu(\cH)$ for every $\cH\in \mathcal{C}$. These properties are often useful in the design of approximation algorithms using a primal-dual approach (see for instance~\cite{GW97,FHRV07}). 

\medskip

Given a graph $G=(V,E)$, an integer $r\ge 0$, and a vertex $v\in V$, we denote by $B_r(v)$ the \emph{ball of radius $r$} in $G$ centered in $v$, that is\[ B_r(v) \coloneqq \sset{u\in V(G)}{d_G(u,v)\le r},\] where $d_G(u,v)$ denotes the distance between $u$ and $v$ in $G$ (we will omit the subscript $G$ when the graph is clear from the context).
We say that a hypergraph $\cH$ is a \emph{ball hypergraph} of $G$ if $\cH$ has vertex set $V=V(G)$ and each edge of $\cH$ is a ball $B_r(v)$ in $G$ for some integer $r$ and some vertex $v\in V$. If all the balls forming the edges of $\cH$ have the same radius $r$, we say that $\cH$ is an \emph{$r$-ball hypergraph} of $G$.

\medskip

Chepoi, Estellon, and Vax{{\`e}}s~\cite{CEV07} proved the existence of a universal constant $\rho$ such that for every $r\ge 0$ and every planar graph $G$ of diameter at most $2r$, the vertices of $G$ can be covered with at most $\rho$ balls of radius $r$. This result was extended to graphs embeddable on a fixed surface with a bounded number of apices in~\cite{BC14}.
Note that $G$ has diameter at most $2r$ if and only if there are no two disjoint balls of radius $r$ in $G$. 
Also, a set of balls of radius $r$ in $G$ covers all of $V(G)$ if and only if their centers intersect all balls of radius $r$ in $G$. 
Thus, these results state equivalently the existence of a universal constant $\rho$ such that for every $r\ge 0$ and every planar (or more generally bounded genus) graph $G$, if the $r$-ball hypergraph $\cH$ consisting of all balls of radius $r$ satisfies $\nu(\cH)=1$, then $\tau(\cH)\le \rho$. 
With this interpretation in mind, Chepoi, Estellon, and Vax{{\`e}}s~\cite{Cam10} conjectured the following generalization in 2007 (see also~\cite{Est07}).

\begin{conj}[Chepoi, Estellon, and Vax{{\`e}}s~\cite{Cam10}]
\label{conj:chep}
There exists a constant $c$ such that for every integer $r\geq 0$, every planar graph $G$, and every $r$-ball hypergraph $\cH$ of $G$, we have $\tau(\cH)\le c\cdot\nu(\cH)$.
\end{conj}

If one considers all metric spaces obtained as standard graph-metrics of planar graphs, then Conjecture~\ref{conj:chep} states that these metric spaces satisfy the so-called {\em bounded covering-packing property}~\cite{CEN17}. 
Recently, Chepoi, Estellon, and Naves~\cite{CEN17} showed that other metric spaces do have this property, including the important case of Busemann surfaces. 
(Quoting~\cite{CEN17}, the latter are roughly the geodesic metric spaces homeomorphic to $\mathbb{R}^2$ in which the distance function is convex; 
they generalize Euclidean spaces, hyperbolic spaces, Riemannian manifolds of global nonpositive sectional curvatures, and CAT(0) spaces.) 

Going back to Conjecture~\ref{conj:chep}, let us emphasize that a key aspect of this conjecture is that the constant $c$ is independent of the radius $r$. 
If $c$ is allowed to depend on $r$, then the conjecture is known to be true. 
In fact, it holds more generally for all graph classes with bounded expansion, as shown by Dvo\v r\'ak~\cite{Dvo13}. 

Some evidence for Conjecture~\ref{conj:chep} was given by Bousquet and Thomass\'e~\cite{BoTh}, who proved that it holds with a polynomial bound instead of a linear one. 
More generally, they proved that for every integer $t \geq 1$, there exists a constant $c_t$ such that for every integer $r\geq 0$, every $K_t$-minor free graph $G$, and every $r$-ball hypergraph $\cH$ of $G$, we have $\tau(\cH)\le c_t\cdot\nu(\cH)^{2t+1}$.

\medskip

The main result of this paper is that Conjecture~\ref{conj:chep} is true, and furthermore it is not necessary to assume that all the balls have the same radius.
\begin{thm}[Main result]
\label{thm:main}
For every integer $t\geq 1$, there is a constant $c_t$ such that $\tau(\cH)\le c_t \cdot \nu(\cH)$ for every $K_t$-minor-free graph $G$ and every ball hypergraph $\cH$ of $G$.
\end{thm}

A set $S$ of vertices of a graph $G$ is \emph{$r$-dominating} if each vertex of $G$ is at distance at most $r$ from $S$, and \emph{$r$-independent} if any two vertices of $S$ are at distance at least $2r+1$ apart in $G$. Note that if we take  $\cH$ to be the $r$-ball hypergraph consisting of all balls of radius $r$ in $G$, Theorem~\ref{thm:main} has the following interesting graph-theoretic interpretation: if $G$ is $K_t$-minor-free, then the minimum size of an $r$-dominating set is at most $c_t$ times the maximum size of an $r$-independent set in $G$.

\medskip

Our proof of Theorem~\ref{thm:main} follows a bootstrapping approach. 
It relies on the existence of some function $f_t$ such that $\tau(\cH)\le f_t(\nu(\cH))$, i.e.\ on the Erd\H os-P\'osa property of the ball hypergraphs of $K_t$-minor-free graphs, which is used in the proof when $\nu(\cH)$ is not `too big'. 
However, showing this property was an open problem. 
This was known for $r$-ball hypergraphs, by the result of Bousquet and Thomass\'e~\cite{BoTh}, but their proof method does not extend to the case of balls of arbitrary radii. 
For this reason, as a first step towards proving Theorem~\ref{thm:main}, we prove Theorem~\ref{thm:nearlinear} below establishing said Erd\H os-P\'osa property. 
We also note that, while the bounding function in Theorem~\ref{thm:nearlinear} is not optimal, it is a near linear bound of the form $\tau(\cH)\le c_t \cdot \nu(\cH) \log \nu(\cH)$ where $c_t$ is a small explicit constant polynomial in $t$. 
This is in contrast with the constant $c_t$ in our proof of Theorem~\ref{thm:main} which is large, exponential in $t$. 
Thus, the bound in Theorem~\ref{thm:nearlinear} is better for small values of $\nu(\cH)$.   
(We note that logarithms in this paper are natural, and the base of the natural logarithm is denoted by $\ee$.)  

\begin{thm}[Near linear bound]
\label{thm:nearlinear}
  Let $G$ be a graph with no $K_t$-minor and such that every minor of $G$ has average degree at most $d$. Then for every ball hypergraph $\cH$ of $G$, $$\tau(\cH)\le 2\ee (t-1) \, d \cdot\nu(\cH)\cdot\log(11\ee \, d \cdot \nu(\cH)).$$ In particular, $\tau(\cH) \le c t^2\sqrt{\log t} \cdot \nu(\cH)\cdot \log(t \, \cdot \nu(\cH))$ for some absolute constant $c>0$, and   if $G$ is planar then $\tau(\cH) \le 48\,\ee  \cdot \nu(\cH)\cdot \log(66\,\ee \, \cdot \nu(\cH))$.
\end{thm}

The proof of Theorem~\ref{thm:nearlinear} uses known results on the VC-dimension of ball hypergraphs of $G$ when $G$ excludes a minor, together with classical bounds relating $\tau(\cH)$ and $\tau^*(\cH)$ when $\cH$ has bounded VC-dimension, as well as the following theorem. 

\begin{thm}[Fractional version]\label{thm:frac}
Let $G$ be a graph and let $d$ be the maximum average degree of a minor of $G$. Then for every ball hypergraph $\cH$ of $G$, we have $\nu^*(\cH)\le \ee \,d \cdot \nu(\cH)$. \\
In particular, if $G$ is planar then $\nu^*(\cH) \le 6\ee  \cdot \nu(\cH)$ and if $G$ has no $K_t$-minor then $\nu^*(\cH) \le c \cdot t\sqrt{\log t} \cdot \nu(\cH)$, for some absolute constant $c>0$.
\end{thm}

We note that results on the VC-dimension of ball hypergraphs in graphs excluding a minor have also been  used recently to obtain improved algorithms for the computation of the diameter in sparse graphs~\cite{DHV20, LP19}.

\smallskip

The proofs of Theorems~\ref{thm:main}, \ref{thm:nearlinear}, and~\ref{thm:frac} are constructive, and can be transformed into efficient algorithms producing transversals (in the case of  Theorems~\ref{thm:main} and \ref{thm:nearlinear}) or matchings (in the case of Theorem~\ref{thm:frac}) of the desired size.

\smallskip

The paper is organized as follows.  
Sections~\ref{sec:struct} and~\ref{sec:tech} are devoted to technical lemmas that will be used in our proofs. 
Theorems~\ref{thm:nearlinear} and~\ref{thm:frac} are proved in Section~\ref{sec:frac}. 
Theorem~\ref{thm:main} is proved in Section~\ref{sec:main}. 
Finally, we conclude the paper in Section~\ref{sec:conclusion} with a construction suggesting that Theorem~\ref{thm:main} does not extend way beyond proper minor-closed classes.

\section{Hypergraphs, balls, and minors}\label{sec:struct}

We will need two technical lemmas, whose proofs are very similar to the proof of~\cite[Theorem 4]{BoTh} and~\cite[Proposition 1]{CEV07}. We start with Lemma~\ref{lem:minor1}, which will be used in the proofs of Theorem~\ref{thm:frac} and Theorem~\ref{thm:main}. We first need the following definitions.

We say that two balls $B_1$ and $B_2$ are \emph{comparable} if  $B_1 \subset B_2$ or $B_2 \subset B_1$ (where $\subset$ denotes the strict inclusion), and otherwise they
are \emph{incomparable}. Note in particular that if $B_1=B_2$ then $B_1$ and $B_2$ are incomparable. Consider two intersecting and incomparable balls $B_1=B_{r_1}(v_1)$ and $B_2=B_{r_2}(v_2)$ in a graph $G$, and let $d:=d_G(v_1,v_2)$. A \emph{median vertex} of $B_1$ and $B_2$ is any vertex $u$ lying on a shortest path between $v_1$ and $v_2$, at distance $\lfloor\frac{r_1-r_2+d}{2}\rfloor$ from $v_1$ and at distance $\lceil\frac{r_2-r_1+d}{2}\rceil$ from $v_2$, or symmetrically at distance $\lceil\frac{r_1-r_2+d}{2}\rceil$ from $v_1$ and at distance $\lfloor\frac{r_2-r_1+d}{2}\rfloor$ from $v_2$. Since $B_1$ and $B_2$ intersect, we have $r_1+r_2\ge d$  and since $B_1$ and $B_2$ are incomparable, we have $r_2\le r_1+d$ and  $r_1\le r_2+d$, and in particular $\lfloor\frac{r_1-r_2+d}{2}\rfloor\ge 0$ and $\lceil\frac{r_2-r_1+d}{2}\rceil\ge 0$ (so the distances above are well defined). 
Moreover, $\lfloor\frac{r_1-r_2+d}{2}\rfloor = \lfloor\frac{2r_1-r_1-r_2+d}{2}\rfloor \le r_1$ and $\lceil\frac{r_2-r_1+d}{2}\rceil = \lceil\frac{2r_2-r_1-r_2+d}{2}\rceil \le r_2$, so any median vertex of $B_1$ and $B_2$ lies in $B_1\cap B_2$. Finally, note that by the definition of a median vertex $u$ of $B_1$ and $B_2$,

\begin{itemize}
\item for every $\{i,j\}=\{1,2\}$  we have $r_j-d(v_j,u)\le r_i-d(v_i,u)+1$, and
\item if  $v_1= v_2$ (which implies $B_1=B_2$ since the balls are incomparable), then $u=v_1=v_2$.
\end{itemize}

\medskip

\begin{lemma}~\label{lem:both1}
\label{lem:minor1}
Let $G$ be a graph, let $S=\{B_i = B_{r_i}(s_i)\}_{i\in [n]}$ be a set of  $n$ pairwise incomparable balls in $G$, with pairwise distinct centers, and let $E_S \subseteq \binom{S}{2}$ be a subset of pairs of intersecting balls $\{B_i,B_j\} \subseteq S$, each of which is associated with a median vertex $x_{\{i,j\}}$ of $B_i$ and $B_j$, and such that the only balls of $S$ containing $x_{\{i,j\}}$ are $B_i$ and $B_j$. Then the graph $H=(S,E_S)$ is a minor of $G$. 
\end{lemma}

\begin{proof}
Let us fix a total ordering $\prec$ on the vertices of $G$. In the proof, all distances are in the graph $G$, so we write $d(u,v)$ instead of $d_G(u,v)$ for the sake of readability.
For every pair of balls $\{B_i,B_j\} \in E_S$, we write $x_{ij}$ or $x_{ji}$ instead of $x_{\{i,j\}}$, for the sake of readability ($x_{ij}$, $x_{ji}$, and $x_{\{i,j\}}$ all correspond to the same median vertex of $B_i$ and $B_j$). We also let $P(s_i, x_{ij})$ be a shortest path from $s_i$ to $x_{ij}$, and we assume that the sequence of vertices from $s_i$ to $x_{ij}$ on the path is minimum with respect to the lexicographic order induced by $\prec$ (among all shortest paths from $s_i$ to $x_{ij}$). By the assumptions, we know that $P_{ij} \coloneqq P(s_i, x_{ij})  \cup P(s_j, x_{ij}) $ is a shortest path from $s_i$ to $s_j$. 

For every $i\in [n]$, we define
\[
\cT_i \coloneqq \bigcup_{j\,:\,\{B_i,B_j\} \in E_S} P(s_i, x_{ij}) .\]

\paragraph{\bf Claim 1.} For every $i\in [n]$, $\cT_i$ is a tree.

\medskip

\noindent Assume for the sake of contradiction that there is a cycle $C$ in $\cT_i$. 
Observe that, by construction, if $uv$ is an edge of $\cT_i$ then $|d(s_i,u) - d(s_i,v)|=1$. 
Let $y$ be a vertex of $C$ maximizing $d(s_i,y)$, and let $z_1,z_2$ denote its two neighbors in $C$. 
Then $d(s_i,z_1)=d(s_i,z_2)=d(s_i,y)-1$, and there exist $j_1,j_2$ such that $z_1y$ is an edge of $P(s_i,x_{ij_1})$ and $z_2y$ is an edge of $P(s_i,x_{ij_2})$. 
Let $P_1$ and $P_2$ be the subpaths from $s_i$ to $y$ of $P(s_i,x_{ij_1})$ and $P(s_i,x_{ij_2})$, respectively. Then $P_1$ and $P_2$ are two different paths from $s_i$ to $y$, and one of them is not minimum either in terms of length, or with respect to the lexicographic order induced by $\prec$. This contradicts the definition of $P(s_i,x_{ij_1})$ and $P(s_i,x_{ij_2})$.

\medskip

\paragraph{\bf Claim 2.}
For every two pairs of balls $\{B_i,B_k\},\{B_j,B_\ell\} \in E_S$ with $i\ne j$, if $P(s_i,x_{ik})$ and $P(s_j,x_{j\ell})$ intersect in some vertex $y$ such that $d(y,x_{ik}) \le d(y,x_{j\ell })$, then $j=k$ and $y=x_{ij}$.

\medskip

\noindent 
Note that $d(s_j,x_{ik})\le d(s_j,y) + d(y,x_{ik}) \le d(s_j,y) + d(y,x_{j\ell }) = d(s_j,x_{j\ell}) $.
Since $x_{j\ell}$ is a median vertex of $B_j$ and $B_\ell$, we have $d(s_j,x_{j\ell}) \le  r_j$, which implies that $d(s_j,x_{ik})\le r_j$ and thus $x_{ik} \in B_j$. 
By definition, $x_{ik}$ is only contained in the balls $B_i$ and $B_k$ of $S$ and thus $j=k$. If we also have $i=\ell$, then necessarily $y=x_{ij}$. 

From now on, we assume that $i \neq \ell$.
Since $P_{ij} = P(s_i, x_{ij})  \cup P(s_j, x_{ij}) $ is a shortest path containing the vertex $y$, the $s_j$--$y$ section of that path (which contains $x_{ij}$) has the same length as the $s_j$--$y$ section of $P(s_j,x_{j\ell})$. 
Replacing the latter section by the former, we obtain a shortest path from $s_j$ to $x_{j\ell}$ containing $x_{ij}$, which we denote $Q(s_j,x_{j\ell})$. 
As a consequence,
\[
  d(x_{j\ell}, x_{ij})=d(x_{j\ell},s_j)-d(s_j,x_{ij})\le r_j-d(s_j,x_{ij})\le r_i-d(s_i,x_{ij})+1, 
\]
where the last inequality follows from the definition of $x_{ij}$. 
We now use the fact that $y$ appears on the path $P(s_i,x_{ij})$ and on the $x_{ij}$--$x_{j\ell}$ section of $Q(s_j,x_{j\ell})$, and obtain
\begin{align*}
d(s_i,x_{j\ell}) &\le d(s_i,y) + d(y,x_{j\ell}) = d(s_i,x_{ij}) + d(x_{ij},x_{j\ell}) - 2 d(y,x_{ij}) \le r_i + 1 - 2 d(y,x_{ij}).
\end{align*}
Since $x_{j\ell} \notin B_i$ by definition (and so $d(s_i,x_{j\ell})>r_i$), this implies that $y=x_{ij}$, as desired.

%

\medskip

This claim immediately implies that for every $i,j\in [n]$ with $i \neq j$, we have $V(\cT_i) \cap V(\cT_j)=\{x_{ij}\}$ if $\{B_i,B_j\} \in E_S$, and $V(\cT_i) \cap V(\cT_j)=\emptyset$ otherwise. Another consequence is that for every $\{B_i,B_j\} \in E_S$, the vertex $x_{ij}$ is a leaf in at least one of the two trees $\cT_i$ and $\cT_j$ (since otherwise there exist $k \neq j$ and $\ell \neq i$ such that $x_{ij} \in P(s_i,x_{ik})$ and $x_{ij} \in P(s_j,x_{j \ell})$, which readily contradicts Claim 2 above).

In the subgraph $\bigcup_{i\in [n]}\cT_i$ of $G$, for each $i\in [n]$ we contract each edge of $\cT_i$ except the ones incident to a leaf of $\cT_i$. It follows from the paragraph above that the resulting graph is precisely a graph obtained from $H=(S,E_S)$ by subdividing each edge at most once, and thus $H$ is a minor of $G$. 
\end{proof}

The next result has a very similar proof\footnote{Despite our best effort, we have not been able to prove the two results at once in a satisfactory way, i.e.\ with a proof that would be both readable and shorter than the concatenation of the two existing proofs.}, but the setting is slightly different. It will be used in the proof of Theorem~\ref{thm:main}.

\begin{lemma}
\label{lem:minor2}
Let $G$ be a graph and $S=\{B_i = B_{r_i}(s_i)\}_{i\in [n]}$ be a set of $n$ pairwise vertex-disjoint  balls in $G$, and let $E_S \subseteq \binom{S}{2}$ be a subset of pairs of balls $\{B_i,B_j\} \subseteq S$, each of which is associated with a ball $B_{\{i,j\}} \notin S$ of $G$ which intersects only $B_i$ and $B_j$ in $S$. Then the graph $H=(S,E_S)$ is a minor of $G$. 
\end{lemma}

\begin{proof}
  Let us fix a total ordering $\prec$ on the vertices of $G$. As before, all distances are in the graph $G$, and we write $d(u,v)$ instead of $d_G(u,v)$. For every $\{B_i,B_j\} \in E_S$ we write $B_{ij}$ or $B_{ji}$ interchangeably for $B_{\{i,j\}}$, and we denote by $x_{ij}$ the center of the ball $B_{ij}$, and by $r_{ij}$ its radius ($x_{ij}=x_{ji}$ and $r_{ij}=r_{ji}$). We can assume that the centers $x_{ij}$ are chosen so that the radii $r_{ij}$ are minimal (among all balls of $G$ not in $S$ that intersect only $B_i$ and $B_j$ in $S$).

  We let $P(s_i,x_{ij})$ be the shortest path from $s_i$ to $x_{ij}$ which minimizes the sequence of vertices from $s_i$ to $x_{ij}$ with respect to the lexicographic ordering induced by $\prec$ (among all shortest paths from $s_i$ to $x_{ij}$). Observe that $P(s_i,x_{ij})$ and $P(s_j,x_{ij})$ only intersect in $x_{ij}$ (if not, we could replace $x_{ij}$ by a vertex that is closer to $s_i$ and $s_j$ and reduce the radius $r_{ij}$ accordingly -- the new ball $B_{ij}$ would still intersect $B_i$ and $B_j$, and no other ball of $S$, and this would contradict the minimality of $r_{ij}$). We may also assume that $r_i+r_{ij}-1\le d(s_i,x_{ij})\le r_i+r_{ij}$ (otherwise we could replace $x_{ij}$ by its neighbor on $P(s_j,x_{ij})$ and decrease $r_{ij}$ by 1).

For every $i\in [n]$, we define
\[
  \cT_i \coloneqq \bigcup_{j\,:\,\{B_i,B_j\} \in E_S} P(s_i,x_{ij}).\]

\medskip

\paragraph{\bf Claim 1.} For every $i\in [n]$, $\cT_i$ is a tree.

\smallskip

\noindent The proof is exactly the same as that of Claim 1 in the proof of Lemma~\ref{lem:minor1} (we do not repeat it here).

\medskip

On the path $P(s_i,x_{ij})$, we let $z_{i ,i j}$ be the vertex at distance $r_i$ from $s_i$ (and  since $x_{ij}=x_{ji}$ we use  $z_{i , ij}$ and $z_{i , ji}$ interchangeably).
As we assumed above that $r_i+r_{ij}-1\le d(s_i,x_{ij})\le r_i+r_{ij}$, we also have 
$r_{ij}-1 \le  d(x_{ij},z_{i,i j})\le r_{ij}$.
In particular,  $ d(x_{ij},z_{j,ij})-1 \le  d(x_{ij},z_{i,i j})\le  d(x_{ij}, z_{j, ij})+1$.

\medskip

\paragraph{\bf Claim 2.} For every two pairs of balls $\{B_i,B_k\},\{B_j,B_\ell\} \in E_S$, with $i\ne j$, if $P(s_i,x_{ik})$ and $P(s_j,x_{j\ell})$ intersect in some vertex $y$ such that $d(y,z_{i ,i k}) \le d(y,z_{j,j \ell})$, then $i=\ell$ and $y=x_{ij}$.

\medskip

We first argue that $y$ appears after $z_{j ,j\ell}$ when traversing $P(s_j,x_{j\ell})$ from $s_j$ to $x_{j\ell}$. Indeed, otherwise we would have
\[ d(s_j,z_{i ,i k}) \le  d(s_j,y) + d(y,z_{i ,i k}) \le d(s_j,y) + d(y,z_{j ,j \ell}) = d(s_j,z_{j ,j \ell})= r_j,\]
which means that $B_i$ and $B_j$ intersect, contradicting the assumptions that $i\neq j$ and all balls in $S$ are vertex-disjoint. 
So $y$ lies on the $z_{j,j\ell}$--$x_{j\ell}$ section of $P(s_j,x_{j\ell})$, and we infer that
\[ d(x_{j\ell},z_{i ,i k}) \le  d(x_{j\ell},y) + d(y,z_{i ,i k})\le d(x_{j\ell},y) + d(y,z_{j ,j \ell}) = d(x_{j\ell},z_{j ,j \ell})\le r_{j\ell}.
\]
It follows that the ball $B_{j\ell}$ intersects the ball $B_i$. By the assumption, this means that $i=\ell$, and thus $s_\ell=s_i$ and $z_{j ,j \ell}=z_{j , ij}$. 
We now argue that $y$ lies in the $z_{i,ik}$--$x_{ik}$ section of $P(s_i, x_{ik})$. 
Suppose for a contradiction that $y$ appears strictly before $z_{i,ik}$ when traversing $P(s_i,x_{ik})$ from $s_i$ to $x_{ik}$.
By definition of $z_{i,ik}$, it then follows that $d(s_i, y) \leq r_i-1$. On the other hand 
$$d(s_j,y) = d(s_j,x_{ij})- d(y,x_{ij}) = d(s_j,z_{j,ij}) + d(z_{j,ij},x_{ij})-d(y,x_{ij}).$$

Note that $r_i+r_{ij}-1\le d(s_i,x_{ij})\le d(s_i,y)+d(y,x_{ij})\le r_i-1 +d(y,x_{ij})$, and thus $d(y,x_{ij})\ge r_{ij}\ge d(z_{i,ij},x_{ij})$. We obtain that  $d(z_{j,ij},x_{ij}) \leq d(z_{i,ij},x_{ij})+1 \leq d(y, x_{ij})+1$, and it follows that $d(s_j, y) \leq d(s_j,z_{j,ij})+1 = r_j+1$. 
Hence $d(s_i,s_j) \leq d(s_i,y) + d(y, s_j) \leq r_i+r_j$, so $B_i$ and $B_j$ intersect, a contradiction. 
We conclude that $y$ lies in the $z_{i,ik}$--$x_{ik}$ section of $P(s_i, x_{ik})$, and thus $d(x_{ik},y)+d(y, z_{i ,i k})= d(x_{ik}, z_{i,ik})$.

Recall that by the initial assumption of the claim, combined with $i=\ell$, we have $d(y,z_{i ,i k}) \le d(y,z_{j , ij})$. Assume first that $d(y,z_{i ,i k}) =d(y,z_{j , ij})$. Then
\[d(x_{ik},z_{j , ij})\le d(x_{ik},y)+d(y, z_{j , ij}) = d(x_{ik},y)+d(y, z_{i ,i k})\le r_{ik},
\]
which implies that $B_j$ intersects $B_{ik}$. 
Thus $j=k$, $P(s_i,x_{ik})=P(s_i,x_{ij})$, and $P(s_j,x_{j\ell})=P(s_j,x_{ij})$. Since these two paths have only $x_{ij}$ in common, in this case we conclude that $y=x_{ij}$.
We can now assume that $d(y,z_{i ,i k}) \le d(y,z_{j , ij}) -1$. Recall that by definition of $x_{ij}$, we have $d(x_{ij},z_{i ,i j})\ge d(x_{ij},z_{j , ij})-1$, which implies that
\[
d(y,z_{i ,i k})+d(y,x_{ij})\le d(y,z_{j , ij})-1+d(y,x_{ij})= d(z_{j , ij},x_{ij})-1\le d(z_{i , ij},x_{ij}).
\]

Since $z_{i,ik}$ and $z_{i,ij}$ are both at distance $r_i$ from $s_i$ and $P(s_i,x_{ij})$ is a shortest path from $s_i$ to $x_{ij}$, it follows that the concatenation of the $s_i$--$y$ section of $P(s_i,x_{ik})$ and the $y$--$x_{ij}$ section of $P(s_j,x_{ij})$ is a shortest path from $s_i$ to $x_{ij}$ (containing $y$). As $y$ is also on a shortest path from $s_j$ to $x_{ij}$, if we had $d(y,x_{ij})>0$, then we could replace $x_{ij}$ by $y$ and reduce $r_{ij}$ to $r_{ij}-d(y,x_{ij})$ ($B_{ij}$ would still intersect $B_i$ and $B_j$ and only these balls of $S$), which would contradict the minimality of $r_{ij}$. It follows that $y=x_{ij}$, as desired.

\medskip

As in the proof of Lemma~\ref{lem:minor1}, the claim implies that for $i \neq j\in [n]$, $\cT_i \cap \cT_j=\{x_{ij}\}$ if $\{B_i,B_j\} \in E_S$, and otherwise the trees $\cT_i$ and $\cT_j$ are vertex-disjoint. Another direct consequence is that for every $\{B_i,B_j\}\in E_S$, the vertex $x_{ij}$ is a leaf in at least one of the two trees $\cT_i$ and $\cT_j$. As before, we can contract the edges of each tree $\cT_i$ not incident to a leaf of $\cT_i$, and the resulting graph is precisely a graph obtained from $H=(S,E_S)$ by subdividing each edge at most once, and thus $H$ is a minor of $G$.
\end{proof}


\section{Hypergraphs and density}\label{sec:tech}

A \emph{partial hypergraph} of $\cH$ is a hypergraph obtained from $\cH$ by removing a (possibly empty) subset of the edges. In addition to hypergraphs, it will also be convenient to consider \emph{multi-hypergraphs}, i.e.\ hypergraphs $\cH=(V,\mathcal{E})$ where $\mathcal{E}$ is a \emph{multiset} of edges.
The \emph{rank} of a hypergraph or multi-hypergraph $\cH$ is the maximum cardinality of an edge of $\cH$.

\medskip

We start with a useful tool, inspired by~\cite{FP10} (see also~\cite{CEM16}), itself inspired by the Crossing lemma.  
Given a graph $G=(V,E)$, we denote by $\mathrm{ad}(G)$ the average degree of $G$, that is $\mathrm{ad}(G)=2|E|/|V|$.

\begin{lemma}\label{lem:hredu}
  Let $\cH=(V,\mathcal{E})$ be a multi-hypergraph of rank at most $k\ge 2$
on $n$ vertices, and
let $E\subseteq {V\choose 2}$ be a set of pairs of vertices $\{u,v\}$
of $V$ such that there exists an edge $e_{uv}$ of $\cH$ containing $u$
and $v$. (Note that we allow that $e_{uv}=e_{xy}$ for two different
pairs $\{u,v\}$ and $\{x,y\}$.) Then the graph $(V,E)$ contains a subgraph $H$
such that $\mathrm{ad}(H)\ge \tfrac{2|E|}{n\ee k}$ and for every edge
$uv$ of $H$, the corresponding edge $e_{uv}$ of $\cH$ contains no
vertex from $V(H)-\{u,v\}$.
\end{lemma}

\begin{proof}
Let $\bH$ be the (random) graph  obtained by selecting each vertex of $\cH$ independently with probability $1/k$, and keeping a single edge (of cardinality 2) between $u$ and $v$ whenever the only selected vertices of $e_{uv}$ are $u$ and $v$. Then we have
\begin{align*}
\esp{|V(\bH)|} &= \frac{n}{k}, \quad \et \\
\esp{|E(\bH)|} &\ge |E|\cdot \frac{1}{k^2}\pth{1-\frac{1}{k}}^{k-2} > \frac{|E|}{\ee k^2},
\end{align*}
since $k\ge 2$. It follows that $\esp{2|E(\bH)|-\frac{2|E|}{n\ee k}|V(\bH)|}> 0$. In particular, there exists a subgraph $H$ of $(V, E)$ such that $\mathrm{ad}(H)\ge \tfrac{2|E|}{n\ee k}$ and for every edge $uv$ of $H$, the edge $e_{uv}$ of $\cH$ contains no
vertex from $V(H)-\{u,v\}$, as desired. 
\end{proof}

The proof actually gives a randomized algorithm producing the graph $H$. This algorithm can easily be derandomized using the method of conditional expectations, giving a deterministic algorithm running in time $O(|E|+n)$.

\medskip

Given a hypergraph $\cH$ and  a matching $\mathcal{B}$  in $\cH$, we define the \emph{packing-hypergraph} $\mathcal{P}(\cH,\mathcal{B})$ as the hypergraph with vertex set $\mathcal{B}$, in which a subset $\mathcal{B}'\subseteq \mathcal{B}$
is an edge if some edge of $\cH$  intersects all the edges in $\mathcal{B}'$ and no other edge of $\mathcal{B}$.

\begin{lemma}
  \label{lem:bounded-hyper-edges}
Let $G$ be a graph such that each minor of $G$ has average degree at most $d$, let $\cH$ be a ball hypergraph of $G$, and let $\mathcal{B}$ be a matching of size $n$ in $\cH$. For every integer $k\ge 2$, the number of edges of cardinality at most $k$ in the packing-hypergraph $\mathcal{P}(\cH,\mathcal{B})$ is at most
\[ \pth{1+d\ee k}^{k-1}\cdot n.\]
\end{lemma}

\begin{proof}
  Let $\mathcal{P}'$ be the partial hypergraph of $\mathcal{P}(\cH,\mathcal{B})$ induced by the edges of cardinality at most $k$. Let $H$ be the graph with vertex set $\mathcal{B}$ in which two distinct vertices are adjacent if they are contained in an edge of $\mathcal{P}'$ (i.e.\ an edge of $\mathcal{P}(\cH,\mathcal{B})$ of cardinality at most $k$). Let $m$ be the number of edges of $H$.  Applying Lemma~\ref{lem:hredu} to $\mathcal{P}'$, we obtain a subgraph $H'$ of $H$ of average degree at least $\tfrac{2m}{n\ee k}$, and such that for any pair $x,y$ of adjacent vertices in $H'$, there is an edge of $\mathcal{P}'$ that contains $x$ and $y$ and no other vertex of $H'$. The vertices of $H'$ correspond to a subset $S$ of pairwise disjoint balls of $G$ (since $\mathcal{B}$ is a matching), and each edge of $H'$ corresponds to a ball of $G$ that intersects some pair of balls of $S$ (and does not intersect any other ball of $S$).

  By Lemma~\ref{lem:minor2}, $H'$ is a minor of $G$, so in particular $\tfrac{2m}{n\ee k}\le \mathrm{ad}(H')\le d$, and hence $m \le \frac12 \, d\ee k n$. It follows that $H$ contains a vertex of degree at most $d\ee k$, and the same is true for every induced subgraph of $H$ (since we can replace $\mathcal{B}$ in the proof by any subset of $\mathcal{B}$). As a consequence, $H$ is $\lfloor d\ee k \rfloor$-degenerate. It is a folklore result that $\ell$-degenerate graphs on $n$ vertices have at most $\binom{\ell}{t-1} n$ cliques of size $t$ (see for instance~\cite[Lemma 18]{Woo16}, where the proof gives a linear time algorithm to enumerate all the cliques of size $t$ when $t$ and $\ell$ are fixed),  and hence there are at most
 \[ n\cdot \sum_{i=1}^k \binom{\lfloor d\ee k \rfloor}{i-1} \le n \cdot \pth{1+d\ee k}^{k-1}\]
cliques of size at most $k$ in $H$, which is an upper bound on the number of edges of cardinality at most $k$ in $\mathcal{P}(\cH,\mathcal{B})$.
\end{proof}

Note that the proof gives an $O(n)$ time algorithm enumerating all edges of cardinality at most $k$ in the packing-hypergraph $\mathcal{P}(\cH,\mathcal{B})$, when $k$ and $d$ are fixed (note that since $H$ is $\lfloor d\ee k \rfloor$-degenerate, it contains a linear number of edges, and thus the application of  Lemma~\ref{lem:hredu} takes time $O(n)$).


\section{Fractional packings of balls}\label{sec:frac}

We now prove Theorem~\ref{thm:frac}. The proof is inspired by ideas from~\cite{ReSh96}.

\medskip

{\noindent \emph{Proof of Theorem~\ref{thm:frac}.} 
Let $\cH$ be a ball hypergraph of $G$.  Since $\nu^*(\cH)$ is attained and is a rational number (recall that $\nu^*(\cH)$ is the solution of a linear program with integer coefficients), there exists a multiset $\mathcal{B}$ of $p$ balls of $G$, such that every vertex $v\in V(G)$ is contained in at most $q$ balls of $\mathcal{B}$, and $\nu^*(\cH)=p/q$ (see for instance~\cite{ReSh96}, where the same argument is applied to fractional cycle packings). We may assume that $q$ is arbitrarily large (by taking $k$ copies of each ball of $\mathcal{B}$, with multiplicities, for some arbitrarily large constant $k$), so in particular we may assume that $q\ge 2$. We may also assume that $G$ contains at least one edge (i.e.\ $d\ge 1$), otherwise the result clearly holds. 
Enumerate all the balls in $\mathcal{B}$ as $B_1, B_2, \dots, B_p$ (and recall that since $\mathcal{B}$ is a multiset, some balls $B_i$ and $B_j$ might coincide). 
We may assume that there is no pair of balls $B_i,B_j$ such that $B_i\subset B_j$ (otherwise we can replace $B_j$ by $B_i$ in $\mathcal{B}$, and we still have a fractional matching). It follows that the balls of $\mathcal{B}$ are pairwise incomparable (as defined at the beginning of Section~\ref{sec:struct}). For any two intersecting balls $B_i$ and $B_j$ we define $x_{ij}$ as a median vertex of $B_i$ and $B_j$ (also defined at the beginning of Section~\ref{sec:struct}). Recall that it implies in particular that whenever $B_i$ and $B_j$ intersect,  $x_{ij}\in B_i\cap B_j$, and if $B_i$ and $B_j$ coincide then $x_{ij}$ is the center of $B_i$ and $B_j$.

We let $\mathcal{G}$ be the intersection graph of the balls in $\mathcal{B}$, that is $V(\mathcal{G})=\mathcal{B}$ and two vertices $B_i,B_j\in \mathcal{B}=V(\mathcal{G})$ with $i\neq j$ are adjacent in $\mathcal{G}$ if and only if ${B_i\cap B_j \neq \varnothing}$.  
(In particular, there is an edge linking $B_i$ and $B_j$ when $B_i$ and $B_j$ are two copies of the same ball.) 
Let $m$ be the number of edges of $\mathcal{G}$. 
Let $\mathcal{B}^*$ denote the multi-hypergraph with vertex set $\mathcal{B}$, where for every vertex of $G$ of the form $x_{ij}$ there is a corresponding edge consisting of the balls in $\mathcal{B}$ that contain $x_{ij}$. 
Note that two distinct such vertices could possibly define the same edge, which is why edges in $\mathcal{B}^*$ could have multiplicities greater than $1$.  
The multi-hypergraph $\mathcal{B}^*$ has rank at most $q$ and contains $p$ vertices. Note moreover that the number of pairs of vertices $B_i,B_j$ of $\mathcal{B}^*$ with $i\neq j$ such that there exists an edge of $\mathcal{B}^*$ containing $B_i$ and $B_j$ is precisely $m$.

By Lemma~\ref{lem:hredu} applied to the multi-hypergraph $\mathcal{B}^*$, we obtain a graph $H=(S,E_S)$ satisfying the following properties:
\begin{itemize}
\item $S\subseteq \mathcal{B}$;
\item for each edge $\{B_i,B_j\}\in E_S$, $x_{ij}$ is contained in $B_i$ and $B_j$ but in no other ball from $S$, and 
\item $\mathrm{ad}(H)\ge \tfrac{2m}{p\ee q}$.
\end{itemize}
    
We would like to apply Lemma~\ref{lem:minor1} to $H$ but this is not immediately possible, since some balls of $S$ might coincide (recall that $\mathcal{B}$ is a multiset), and therefore the centers of the balls of $S$ might not be pairwise distinct. However, observe that if two balls of $S$ coincide, then by definition the two corresponding vertices of $H$ have degree either 0 or 1 in $H$ (and in the latter case the two vertices are adjacent in $H$). Indeed, if two balls $B_i,B_j$ of $S$ coincide and $B_i$ is adjacent to $B_k$ in $H$ with $k \neq j$, then the only balls of $S$ containing $x_{ik}$ are $B_i$ and $B_k$, contradicting the fact that $x_{ik}$ is also in $B_j$. 

Let $S_1 \subseteq S$ be the subset of balls of $S$ having multiplicity $1$ in $S$. Since no ball of $\mathcal{B}$ is a strict subset of another ball of $\mathcal{B}$, the centers of the balls of $S_1$ are pairwise distinct. As a consequence of the previous paragraph, if we consider the subgraph $H_1$ of $H$ induced by $S_1 $, then $\mathrm{ad}(H)\le \max(1,\mathrm{ad}(H_1))$.

By Lemma~\ref{lem:minor1} applied to the set of balls $S_1$ in $G$, we obtain that $H_1$ is a minor of $G$ and 
thus $\mathrm{ad}(H_1)\le d$. It follows that  $\tfrac{2m}{p\ee q}\le \mathrm{ad}(H)\le \max(1, d)\le d$ (since $d\ge 1$). This implies that the average degree $2m/p$ of $\mathcal{G}$ is at most
$\ee \,dq$. By the Caro-Wei inequality~\cite{Caro79,Wei81} (or Tur\'an's theorem~\cite{Turan41}), it follows that $\mathcal{G}$ contains an independent set of size at least $$\frac{|V(\mathcal{G})|}{\mathrm{ad}(\mathcal{G})+1}\ge \frac{p}{\ee \,dq+1}=\frac{\nu^*(\cH)}{\ee \,d+1/q}.$$ An independent set in $\mathcal{G}$ is precisely a matching in $\cH$, and thus $\nu(\cH)\ge \tfrac1{\ee\,d+1/q}\cdot\nu^*(\cH)$ and $\nu^*(\cH)\le (\ee\,d+1/q)\cdot\nu(\cH)$. Since we can assume that $q$ is arbitrarily large, it follows that $\nu^*(\cH)\le \ee\,d\cdot\nu(\cH)$, as desired.

\medskip

The rest of the result follows from well known results on the average degree of graphs. On the one hand, an easy consequence of Euler's formula is that planar graphs have average degree at most 6. On the other hand, it was proved by Kostochka~\cite{Kos84} and Thomason~\cite{Tho84} that  every $K_{t}$-minor-free graph has average degree $O(t\sqrt{\log t})$. \hfill $\Box$

\bigskip

The linear program for $\nu^*$ has coefficients in $\{0,1\}$, and can thus be solved in time $O(n^3)$, since we can assume that the balls have pairwise distinct centers (and so the number of variables and inequalities is linear in the number of vertices). The associated rational coefficients $w_e$ can thus be found in time $O(n^3)$. It is then convenient to define $w_e'$ as the largest $\tfrac{\ell}{n}\le w_e$ with $\ell\in \mathbb{N}$. Note that the coefficients $(w_e')$ still satisfy the inequalities of the linear program for $\nu^*$, and their sum is at least $\nu^*-1$ since we can assume that there are at most $n$ balls (since there centers are pairwise distinct).  There is a small loss on the multiplicative constant (compared to the statement of Theorem~\ref{thm:frac}), but we can now assume that in the proof we have $q\le n$ and thus $p\le n^2$ and $m=O(n^3)$. It follows that the application of  Lemma~\ref{lem:hredu} can be done in time $O(m)=O(n^3)$, and the construction of a stable set of suitable size in $\mathcal{G}$ can also be done in time $O(m)=O(n^3)$. Therefore, the proof of Theorem~\ref{thm:frac} gives an $O(n^3)$ time algorithm constructing a matching of size $\Omega(\nu^*(\mathcal{H}))$ in $\mathcal{H}$.

\medskip

The \emph{VC-dimension} of a hypergraph $\cH$ is the cardinality of a largest subset $X$ of vertices such that for every $X'\subseteq X$, there is an edge $e$ in $\cH$ such that $e\cap X=X'$. Bousquet and Thomassé~\cite{BoTh} proved the following result.

\begin{thm}\label{thm:both}
  If $G$ has no $K_t$-minor, then every ball hypergraph $\cH$ of $G$ has VC-dimension at most $t-1$.
\end{thm}

A classical result is that for hypergraphs of bounded VC-dimension, $\tau=O(\tau^*\log \tau^*)$. We will use the following precise bound of Ding, Seymour, and Winkler~\cite{DSW94}.

\begin{thm}\label{thm:dsw}
If a hypergraph $\cH$ has VC-dimension at most $\delta$, then \[\tau(\cH)\le 2 \delta \tau^*(\cH) \log(  11\tau^*(\cH)).\]
\end{thm}

Combining Theorems~\ref{thm:frac}, \ref{thm:both}, and \ref{thm:dsw}, and using that $\nu^*(\cH)=\tau^*(\cH)$, we obtain Theorem~\ref{thm:nearlinear} as a direct consequence.

\medskip

As before, the linear program for $\tau^*$ has coefficients in $\{0,1\}$, and can thus be solved in time $O(n^3)$, since we can assume that the balls have pairwise distinct centers (and so the number of variables and inequalities is linear in the number of vertices). The associated rational coefficients $w_v$ can thus be found in time $O(n^3)$. Using algorithmic versions of Theorem~\ref{thm:dsw} (see~\cite{HW87,Mat95}) and the coefficients $(w_v)$, a transversal of $\mathcal{H}$ of size $O(\tau^* \log \tau^*)=O(\nu \log \nu)$ can be found by a randomized algorithm sampling $O(\tau^* \log \tau^*)$ vertices according to the distribution given by $(w_v)$, or a deterministic algorithm running in time $O(n ({\tau^*}^2 \log \tau^*)^t)$. So the overall complexity of obtaining a transversal of the desired size is $O(n^3)$ (randomized) and $O(n^3+n (\tau^* \log \tau^*)^t)$ (deterministic). In the remainder of the paper, the result will be used when $\tau^*$ is a fixed constant, in which case the complexity of the deterministic algorithm is also $O(n^3)$.

\section{Linear bound}\label{sec:main}

In this section we prove Theorem~\ref{thm:main}. 
Recall that by Theorem~\ref{thm:nearlinear}, there is a (monotone) function $f_t$ such that $\tau(\cH)\le f_t(\nu(\cH))$ for every ball hypergraph $\cH$ of a $K_t$-minor-free graph. In the proof, we write $d_t$ for the supremum of the average degree of $G$ taken over all graphs $G$ excluding $K_t$ as a minor. Recall that $d_t=O(t\sqrt{\log t})$~\cite{Kos84,Tho84}.

\medskip

Let $t\ge 1$ be an integer and let $c_t:=2\cdot (1+\tfrac32 d_t^2\ee )^{3d_t/2}\cdot f_t(\tfrac32 d_t)$. We will prove that every ball hypergraph $\cH$ of a $K_t$-minor-free graph satisfies $\tau(\cH)\le c_t \cdot \nu(\cH)$.

\medskip

{\noindent \emph{Proof of Theorem~\ref{thm:main}.} 
We prove the result by induction on $k\coloneqq \nu(\cH)$. The result clearly holds if $k=0$ so we may assume that $k\ge 1$. If $k\le \tfrac32 d_t$ then by the definition of $f_t$ we have $\tau(\cH)\le f_t(\tfrac32 d_t)\le c_t \le c_t \cdot k$, as desired.

Assume now that $k\ge \tfrac32 d_t$ and for every ball hypergraph $\cH'$ of a $K_t$-minor-free graph with $\nu(\cH')<k$, we have $\tau(\cH')\le c_t \cdot \nu(\cH')$. Let $G$ be a $K_t$-minor-free graph and $\cH$ be a ball hypergraph of $G$ with  $\nu(\cH)=k$. Our goal is to show that $\tau(\cH)\le c_t \cdot k$. Note that we can assume that $\cH$ is \emph{minimal}, in the sense that no edge of $\cH$ is contained in another edge of $\cH$ (otherwise we can remove the larger of the two from $\cH$, this does not change the matching number nor the transversal number). 

\medskip

Consider a maximum matching $\mathcal{B}$ (of cardinality $k$) in $\cH$. Let $\mathcal{E}_1$ be the set consisting of all the edges of $\cH$ that intersect at most $\tfrac32 d_t$ edges of $\mathcal{B}$. Note that each edge of $\mathcal{B}$ lies in $\mathcal{E}_1$, and therefore $\mathcal{E}_1$ is non-empty. By Lemma~\ref{lem:bounded-hyper-edges}, the packing-hypergraph $\mathcal{P}(\cH,\mathcal{B})$ contains at most $(1+\tfrac32 d_t^2\ee )^{3d_t/2}\cdot k$ edges of cardinality at most $\tfrac32 d_t$. For each such edge $e$ of $\mathcal{P}(\cH,\mathcal{B})$, consider the corresponding subset $\mathcal{B}_e$ of at most $\tfrac32 d_t$ edges of $\mathcal{B}$, and the subset $\mathcal{E}_e$ of edges of $\cH$ that intersect each ball of $\mathcal{B}_e$, and no other ball of $\mathcal{B}$. Denoting by $\cH_e$ the  partial hypergraph of $\cH$ with edge set $\mathcal{E}_e$, observe that by the maximality of the matching $\mathcal{B}$ we have $\nu(\cH_e)\le \tfrac32 d_t$  (since in $\mathcal{B}$, replacing the edges of $\mathcal{B}_e$ by a matching of $\mathcal{E}_e$ again gives a matching of $\cH$). It follows that $\tau(\cH_e)\le f(\tfrac32 d_t)$. And thus, if we denote by $\cH_1$ the  partial hypergraph of $\cH$ with edge set $\mathcal{E}_1$, we have
$$\tau(\cH_1)\le (1+\tfrac32 d_t^2\ee )^{3d_t/2}\cdot f(\tfrac32 d_t)\cdot  k = \tfrac12 c_t\cdot k.$$ 

Consider now the subset $\mathcal{E}_2$ consisting of all the edges of $\cH$ that intersect more than $\tfrac32 d_t$ edges of $\mathcal{B}$, and let $\cH_2$ be the  partial hypergraph of $\cH$ with edge set $\mathcal{E}_2$. Note that $\mathcal{E}_1$ and $\mathcal{E}_2$ partition the edge set of $\cH$ and thus $\tau(\cH)\le \tau(\cH_1)+\tau(\cH_2)$. Let $\mathcal{B}_2$ be a maximum matching in $\cH_2$, and let $\ell=\nu(\cH_2)=|\mathcal{B}_2|$.
Let $H$ be the (bipartite) intersection graph of the edges of  $\mathcal{B}\cup \mathcal{B}_2$, i.e.\ each vertex of $H$ corresponds to an edge of $\mathcal{B}\cup \mathcal{B}_2$, and two vertices are adjacent if the corresponding edges intersect. (The graph is bipartite because $\mathcal{B}$ and $\mathcal{B}_2$ are matchings.)

Note that since $H$ is bipartite, for every two distinct edges $\{B,B'\}$ and $\{C,C'\}$ of $H$, the sets $B\cap B'$ and $C\cap C'$ are disjoint. Moreover, no ball of $\mathcal{B}\cup \mathcal{B}_2$ is a subset of another ball of $\mathcal{B}\cup \mathcal{B}_2$, and thus the balls of $\mathcal{B}\cup \mathcal{B}_2$
are pairwise incomparable (as defined at the beginning of Section~\ref{sec:struct}). 
So, enumerating the balls in $\mathcal{B}\cup \mathcal{B}_2$ as $B_1, B_2, \dots, B_n$, we can choose, for each edge $\{B_i,B_j\}$ of $H$, a median vertex $x_{ij}$ of $B_i$ and $B_j$ (also defined at the beginning of Section~\ref{sec:struct}). Recall that $x_{ij}\in B_i\cap B_j$, and thus it follows from the property above that the only balls of $\mathcal{B}\cup \mathcal{B}_2$ containing $x_{ij}$ are $B_i$ and $B_j$. 
By Lemma~\ref{lem:minor1}, $H$ is a minor of $G$ and thus has average degree at most $d_t$. On the other hand, the vertices of $H$ corresponding to the edges of $\mathcal{B}_2$ have degree at least $\tfrac32 d_t$ in $H$, and thus
$$\tfrac32 d_t \cdot \ell\le \tfrac12 \,\mathrm{ad}(H)(k+\ell)\le \tfrac12 d_t \cdot (k+\ell),$$ where the central term counts the number of edges of $H$. It follows that $\nu(\cH_2)=\ell\le \tfrac{k}2$, and thus by the induction hypothesis we have $\tau(\cH_2)\le c_t\cdot \nu(\cH_2)\le c_t \cdot \tfrac{k}2$. As a consequence, $$\tau(\cH)\le \tau(\cH_1)+\tau(\cH_2)\le \tfrac12c_t\cdot k+  c_t \cdot \tfrac{k}2 = c_t \cdot k,$$ which concludes the proof of Theorem~\ref{thm:main}.
\hfill $\Box$

\bigskip

The first part of the proof of Theorem~\ref{thm:main} uses Theorem~\ref{thm:nearlinear} when $\nu$ (and thus $\tau^*$, by Theorem~\ref{thm:frac}) is bounded by a function of the constant $t$, and in this case, by the discussion after the proof of Theorem~\ref{thm:nearlinear}, a transversal of the desired size can be found deterministically in time $O(n^3)$.

The second part of the proof of Theorem~\ref{thm:main} can be made constructive by performing the following small modification. We observe that we have not quite used the fact that $\mathcal{B}$ is a \emph{maximum} matching of $\cH$, simply that it has the property that, for any edge $e$ in the packing-hypergraph $\mathcal{P}(\cH, \mathcal{B})$ of cardinality at most $\tfrac32d_t$, the matching number of $\cH_e$ is bounded.
As we have explained after Lemma~\ref{lem:bounded-hyper-edges}, such edges can be enumerated in linear time when $t$ is fixed. We can then compute each $\tau^*(\mathcal{H}_e)=\nu^*(\mathcal{H}_e)$ in time $O(n^3)$ and if this value is more than $\mathrm{e}\, d_t \cdot |e|$, then we can find a matching of size more than $|e|=|\mathcal{B}_e|$ in $\mathcal{H}_e$ in time $O(n^3)$ by Theorem~\ref{thm:frac}, and replace $\mathcal{B}_e$ by this larger matching in $\mathcal{B}$, thus increasing the size of $\mathcal{B}$ (this can be done at most $\nu(\mathcal{H})$ times). On the other hand, if for all the (linearly many) edges $e$ as above, we have $\tau^*(\mathcal{H}_e)\le \mathrm{e}\,d_t \cdot|e|=O({d_t}^2)$, then by Theorem~\ref{thm:nearlinear}, we can find a transversal of size $O({d_t}^2 \log d_t)$ in each hypergraph $\mathcal{H}_e$ in time $O(n^3)$. So overall we find a matching $\mathcal{B}$ that has the desired property, and a transversal of the partial hypergraph of $\mathcal{H}$ with edge set $\mathcal{E}_1$ of the desired size in time $O(\nu(\mathcal{H})\cdot n^4)$. Taking the induction step into account (which divides $\nu$ by at least 2), we obtain a deterministic algorithm constructing a transversal of size $O(\nu(\mathcal{H}))$ in $\mathcal{H}$, in time $O(\sum_{i\ge 0}\tfrac{1}{2^i}\cdot \nu(\mathcal{H})\cdot n^4)=O(\nu(\mathcal{H})\cdot n^4)$, when $t$ is a fixed constant.

\section{Conclusion}
\label{sec:conclusion}

The proof of Theorem~\ref{thm:main} gives a bound of the order of $\exp(t \log^{3/2}t)$ for the constant $c_t$. It would be interesting to improve this bound to a polynomial in $t$.

\medskip

It is also natural to wonder whether Theorem~\ref{thm:main} remains true in a  setting broader than proper minor-closed classes. Natural candidates are graphs of bounded maximum degree, graphs excluding a topological minor, $k$-planar graphs, classes with polynomial growth (meaning that the size of each ball is bounded by a polynomial function of its radius, see e.g.~\cite{KL03}), and classes with strongly sublinear separators (or equivalently, classes with polynomial expansion~\cite{DN16}). We now observe that in all these cases, the associated ball hypergraphs do not satisfy the Erd\H os-P\'osa property, even if all the balls have the same radius. 
That is, we can find $r$-ball hypergraphs in these classes with bounded $\nu$ and unbounded $\tau$. 
Our construction shows that this is true even in the seemingly simple case of subgraphs of a grid with all diagonals (i.e.\ strong products of two paths). 

\medskip

Fix two integers $k,\ell$ with $k\geq 3$, and $\ell$ sufficiently large compared to $k$ and divisible by $2(\binom{k}{2}-1)$. Given $k$ vertices $v_0,v_1,\ldots v_{k-1}$, an $\ell$-\emph{broom with root $v_0$ and leaves $v_1,\ldots,v_{k-1}$} is a tree $T$ of maximum degree 3 with root $v_0$ and leaves $v_1,\ldots,v_{k-1}$} such that
\begin{enumerate}
\item each leaf is at distance $\ell$ from the root $v_0$,
\item the ball of radius $\ell/2$ centered in $v_0$ in $T$ is a path (called the \emph{handle} of the broom), and
  \item the distance between every two vertices of degree 3 in $T$ is sufficiently large compared to $k$.
\end{enumerate}

We now construct a graph $G_{k,\ell}$ as follows. We start with a set $X$ of $k$ vertices $x_1,\ldots,x_k$, and a path of $\binom{k}{2}$ vertices with vertex set $Y=\{y_{\{i,j\}}\,|\,1\le i<j\le k\}$, disjoint from $X$. We then subdivide each edge of the latter path $\tfrac{\ell}{2} \frac{1}{\binom{k}{2}-1} -1$ times, so that the subdivided path has length $\ell/2$. Finally, for each $1\le i \le k$, we add an $\ell$-broom $T_i$ with root $x_i$ and leaves $Y_i=\{y_{\{i,j\}}\,|\, j\ne i\}$.

\begin{figure}[htb]
 \centering
 \includegraphics[scale=1]{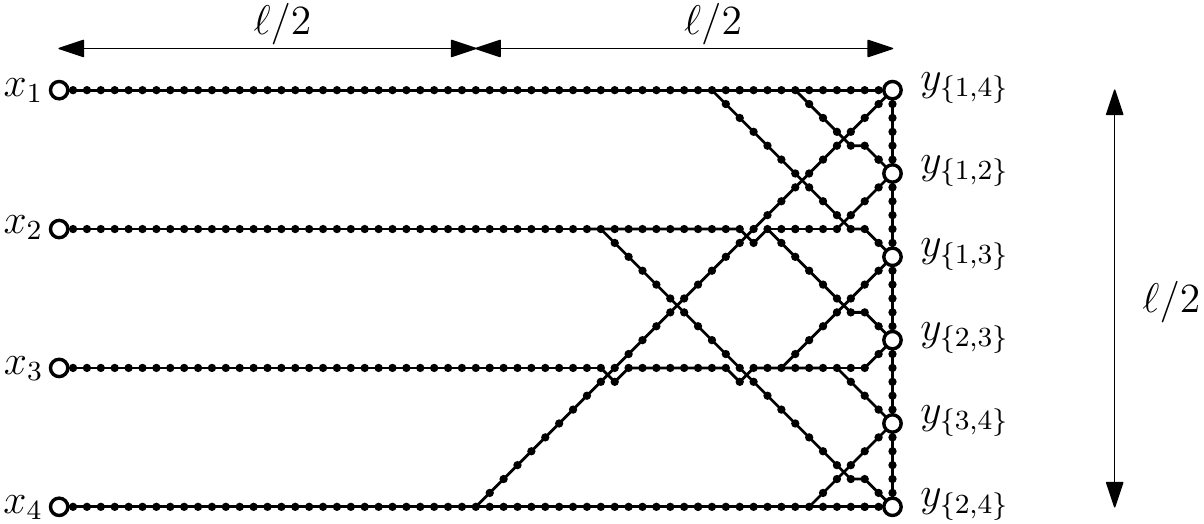}
 \caption{An embedding of the graph $G_{4,\ell}$ in the 2-dimensional grid with all diagonals (the grid itself is not depicted for the sake of clarity).}
 \label{fig:broom}
\end{figure}

We first claim that $G_{k,\ell}$ is a subgraph of the 2-dimensional grid with all diagonals (i.e.\ the strong products of two paths). 
To see this, place $X$ on a single column on the left, and $Y$ on another column on the right (in the sequence given by the path), at distance $\ell$ from the column of $X$, then draw each of the brooms in the plane (with  crossings allowed). Since the distance between two vertices of degree 3 in a broom is sufficiently large compared to  $k$, we can safely embed each topological crossing in the strong product of two edges (see Figure~\ref{fig:broom} for an example).

Let $\cH_{k,\ell}$ be the $\ell$-ball hypergraph of $G_{k,\ell}$ obtained by considering all the balls of radius $\ell$ in $G_{k,\ell}$. We first observe that $\nu(\cH_{k,\ell})=1$: this follows from the fact that each ball of radius $\ell$ centered in a vertex that does not belong to the handle of a broom contains all the vertices of $Y$, while every two vertices on the handles of two brooms $T_i$ and $T_j$ are at distance at most $\ell$ from $y_{\{i,j\}}$. Finally, for every two vertices $x_i$ and $x_j$ of $X$, note that $y_{\{i,j\}}$ is the unique vertex of $G_{k,\ell}$ lying at distance at most $\ell$ from $x_i$ and $x_j$, and thus $\tau(\cH_{k,\ell})\ge \tfrac{k}2$. It follows that there is no function $f$ such that $\tau(\cH)\le f(\nu(\cH))$ for every ball hypergraph of a subgraph of the strong product of two paths (even when all the balls in the ball hypergraph have the same radius).


\begin{acknowledgement} 
We thank the two anonymous reviewers for their detailed
comments and suggestions.
\end{acknowledgement}

\end{document}